\documentclass[reqno]{amsart}
\usepackage{amssymb,amsmath,amsthm,mathrsfs,graphics,hyperref,stmaryrd,psfrag,amscd}
\numberwithin{equation}{section}
\DeclareMathOperator{\Adm}{Adm}
\DeclareMathOperator{\id}{id}

\DeclareMathOperator{\ParP}{Parp} 
\DeclareMathOperator{\rank}{rank}
\DeclareMathOperator{\Sky}{Sky}

\newcommand{\padup}{\rule{0mm}{4mm}}
\newcommand{\0}{\emptyset}
\newcommand{\x}{\times}
\newcommand{\excise}[1]{}
\newcommand{\Case}[1]{\emph{Case~#1:}}

\newcommand{\SR}{SR} 
\newcommand{\sgn}{\varepsilon}
\newcommand{\simij}{\underset{i,j}{\sim}}
\newcommand{\sm}{\setminus}
\newcommand{\st}{\colon}
\newcommand{\Sym}{\mathfrak{S}}
\newcommand{\Alt}{\mathfrak{A}}

\newcommand{\qqandqq}{\qquad\text{and}\qquad}
\newcommand{\biqjr}[4]{b({#1}\uparrow{#2},{#3}\downarrow{#4})} 

\newcommand{\e}{{\mathbf{e}}}
\renewcommand{\H}{{\mathbf{H}}}
\newcommand{\J}{{\mathbf{J}}}
\newcommand{\soff}{\mathbf{w}}  
\newcommand{\F}{{\mathbf{F}}}
\newcommand{\X}{{\mathbf{X}}}
\newcommand{\Y}{{\mathbf{Y}}}
\newcommand{\Z}{{\mathbf{Z}}}
\renewcommand{\P}{{\mathbf{P}}}
\newcommand{\Q}{{\mathbf{Q}}}
\newcommand{\R}{{\mathbf{R}}}
\newcommand{\BB}{\mathcal{B}}

\newcommand{\HH}{\mathcal{H}}

\newcommand{\Zz}{\mathbb{Z}}
\newcommand{\Rr}{\mathbb{R}}

\newtheorem{theorem}{Theorem}[section]
\newtheorem{conjecture}[theorem]{Conjecture}
\newtheorem{proposition}[theorem]{Proposition}
\newtheorem{lemma}[theorem]{Lemma}
\newtheorem{corollary}[theorem]{Corollary}
\theoremstyle{definition}
\newtheorem{definition}[theorem]{Definition}
\newtheorem{example}[theorem]{Example}
\newtheorem{question}[theorem]{Question}

\newcommand{\includefigure}[3]{{
  \begin{center}
  \resizebox{#1}{#2}{\includegraphics{{#3}}}
  \end{center}}}

\author{Jeremy L.\ Martin}
\address{Department of Mathematics\\ University of Kansas\\ Lawrence, KS 66045}
\email{jmartin@math.ku.edu}
\author{Jennifer D.\ Wagner}
\address{Department of Mathematics and Statistics\\ Washburn University\\ Topeka, KS 66621}
\email{jennifer.wagner1@washburn.edu}
\title{On the Spectra of Simplicial Rook Graphs}
\date{April 22, 2014}
\thanks{First author supported in part by a Simons Foundation Collaboration Grant and by National Security Agency grant no.~H98230-12-1-0274.}
\subjclass[2010]{05C50} 
\keywords{graph, simplicial rook graph, integral, spectrum, eigenvalues}
\begin{document}
\begin{abstract}
The \emph{simplicial rook graph} $\SR(d,n)$ is the graph whose vertices are the lattice points in the $n$th dilate of the standard simplex in $\mathbb{R}^d$, with two vertices adjacent if they differ in exactly two coordinates.  We prove that the adjacency and Laplacian matrices of $\SR(3,n)$ have integral spectrum for every~$n$.  The proof proceeds by calculating an explicit eigenbasis.
We conjecture that $\SR(d,n)$ is integral for all~$d$ and~$n$, and present evidence in
support of this conjecture.  For $n<\binom{d}{2}$, the evidence indicates that the smallest eigenvalue of the adjacency matrix is $-n$,
and that the corresponding eigenspace has dimension given by the Mahonian
numbers, which enumerate permutations by number of inversions.
\end{abstract}
\maketitle

\section{Introduction}

Let $d$ and $n$ be nonnegative integers.  The \emph{simplicial rook
graph} $\SR(d,n)$ is the graph with vertices
  $$V(d,n):= \left\{x=(x_1,\dots,x_d) \st 0\leq x_i\leq n,\ \sum_{i=1}^d x_i=n\right\}$$
 with two vertices adjacent if they
agree in all but two coordinates.  This graph has $N=\binom{n+d-1}{d-1}$
vertices and is regular of degree $\delta=(d-1)n$.  Geometrically, let
$\Delta^{d-1}$ denote the standard
simplex in $\Rr^d$ (i.e., the convex hull of the standard basis vectors $\e_1,\dots,\e_d$)
and let $n\Delta^{d-1}$ denote its $n^{th}$ dilate (i.e., the convex hull of $n\e_1,\dots,n\e_d$).
Then $V(d,n)$ is the set of lattice points in $n\Delta^{d-1}$,
with two points adjacent if their difference is a multiple of $\e_i-\e_j$ for some $i,j$.
  Thus the
independence number of $\SR(d,n)$ is the maximum number of nonattacking rooks that
can be placed on a simplicial chessboard with $n+1$ ``squares'' on each side.
Nivasch and Lev \cite{NAQ} and Blackburn, Paterson and Stinson \cite{Dots}
showed independently that for $d=3$, this independence number is
$\lfloor(2n+3)/3\rfloor$.

\begin{figure}[ht]
\includefigure{1.8in}{1.8in}{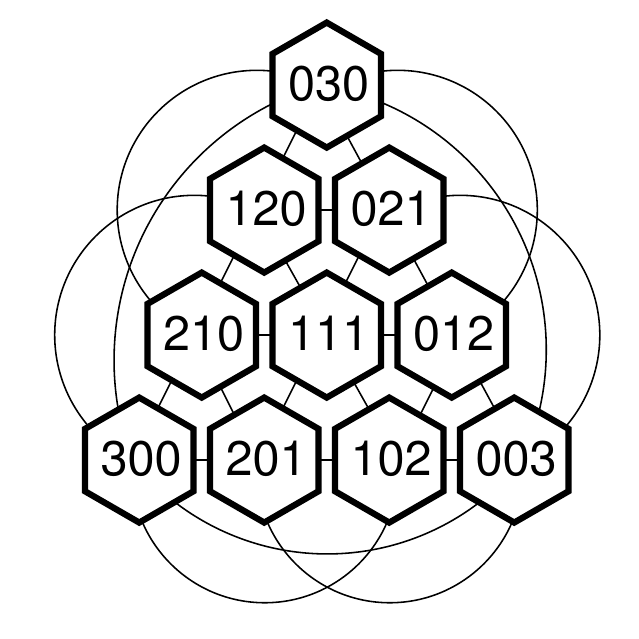}
\caption{The graph $\SR(3,3)$.
\label{\SR-figure}}
\end{figure}

As far as we can tell, the class of simplicial rook graphs has not been studied
before.
For some small values of the parameters, $\SR(d,n)$ is a well-known
graph: $\SR(2,n)$ and $\SR(d,1)$ are complete of orders $n+1$ and~$d$ respectively;
$\SR(3,2)$ is isomorphic to the octahedron; and $\SR(d,2)$ is isomorphic to
the Johnson graph $J(d+1,2)$.
On the other hand, simplicial rook graphs
are not in general vertex-transitive, strongly regular or distance-regular, nor are they line graphs or noncomplete
extended $p$-sums (in the sense of \cite[p.~55]{Cvetkovic}).  They are
also not to be confused with the \emph{simplicial grid graph}, in which
two vertices are adjacent only if their difference vector is exactly $\e_i-\e_j$ (as opposed
to some scalar multiple)
nor with the \emph{triangular graph} $T_n$, which is the line graph of $K_n$
\cite[p.23]{Spectra}, \cite[\S10.1]{GodRoy}.

Let $G$ be a simple graph on vertices $[n]=\{1,\dots,n\}$.
The \emph{adjacency matrix} $A=A(G)$ is the $n\x n$ symmetric matrix
whose $(i,j)$ entry is 1 if $ij$ is an edge, 0 otherwise.
The \emph{Laplacian matrix} is $L=L(G)=D-A$, where~$D$ is
the diagonal matrix whose $(i,i)$ entry is the degree of vertex~$i$.
The graph~$G$ is said to be \emph{integral} (resp.\ \emph{Laplacian integral})
if all eigenvalues of $A$ (resp.\ $L$) are integers. 
If~$G$ is regular of degree~$\delta$, then these conditions
are equivalent, since  every eigenvector
of $A$ with eigenvalue~$\lambda$ is an eigenvector of $L$ with eigenvalue $\delta-\lambda$.  

We can now state our main theorem.

\begin{theorem} \label{main-theorem}
For every $n\geq 1$, the simplicial rook graph $\SR(3,n)$ is integral
and Laplacian integral, with eigenvalues as follows:

\begin{center}
\begin{tabular}{cccc}
\multicolumn{4}{c}{\boldmath\bf If $n=2m+1$ is odd:}\\
{\boldmath\bf Eigenvalue of~$A$} & {\boldmath\bf Eigenvalue of~$L$} & {\bf Multiplicity} & {\bf Eigenvector}\\ \hline
$-3$ & $4m+5=2n+3$ & $\padup\binom{2m}{2}$ & $\H_{a,b,c}$\\
$-2,-1,\dots,m-3$ & $3m+5\dots,4m+4$ & $3$ & $\P_k$\\
$m-1$ & $3m+3$ & $2$ & $\R$\\
$m, \dots, 2m-1=n-2$ & $2m+3\dots,3m+2$ & $3$ & $\Q_k$\\
$4m+2=2n$ & $0$ & $1$ & $\J$
\end{tabular}
\end{center}\bigskip

\begin{center}
\begin{tabular}{cccc}
\multicolumn{4}{c}{\boldmath\bf If $n=2m$ is even:}\\
{\boldmath\bf Eigenvalue of~$A$} & {\boldmath\bf Eigenvalue of~$L$} & {\bf Multiplicity} & {\bf Eigenvector}\\ \hline
$-3$ & $4m+3=2n+3$ & $\padup\binom{2m-1}{2}$ & $\H_{a,b,c}$\\
$-2,-1,\dots,m-4$ & $3m+4,\dots,4m+2$ & $3$ & $\P_k$\\
$m-3$ & $3m+3$ & $2$ & $\R$\\
$m-1,\dots,2m-2=n-2$ & $2m+2,\dots,3m+1$ & $3$ & $\Q_k$\\
$4m=2n$ & $0$ & $1$ & $\J$
\end{tabular}
\end{center}
\end{theorem}

Integrality and Laplacian integrality
typically arise from tightly controlled combinatorial structure in
special families of graphs, including complete graphs, complete bipartite
graphs and hypercubes (classical; see, e.g., \cite[\S5.6]{EC2}), Johnson graphs \cite{Johnson}, Kneser graphs \cite{Lovasz}
and threshold graphs \cite{Merris}.  (General references on graph
eigenvalues and related topics include
\cite{SurveyIG,Spectra,Cvetkovic,GodRoy}.)  For simplicial
rook graphs,  lattice geometry provides this combinatorial structure.
To prove Theorem~\ref{main-theorem}, we construct a basis of
$\Rr^{\binom{n+2}{2}}$ consisting of eigenvectors of~$A(SR(3,n))$, as indicated
in the tables above.  The basis vectors $\H_{a,b,c}$ for the largest
eigenspace (Prop.~\ref{hex-prop}) are signed characteristic vectors for hexagons centered at
lattice points in the interior of $n\Delta^3$ (see
Figure~\ref{hxyz-figure}).  The other eigenvectors $\P_k,\R,\Q_k$
 (Props.~\ref{dim-two-eigenspace}, \ref{appalling-calculation-one}, \ref{appalling-calculation-two})
are most easily
expressed as certain sums of characteristic vectors of lattice lines.

Theorem~\ref{main-theorem}, together with Kirchhoff's matrix-tree theorem
\cite[Lemma~13.2.4]{GodRoy} implies the following formula for the
number of spanning trees of~$\SR(d,n)$.

\begin{corollary} \label{spanning-tree-count}
The number of spanning trees of $\SR(3,n)$ is
\[\begin{cases}
\displaystyle\frac{32(2n+3)^{\binom{n-1}{2}} \prod\limits_{a=n+2}^{2n+2} a^3}{3(n+1)^2(n+2)(3n+5)^3}
& \text{ if $n$ is odd,}
\\\\
\displaystyle\frac{32(2n+3)^{\binom{n-1}{2}} \prod\limits_{a=n+2}^{2n+2} a^3} {3(n+1)(n+2)^2(3n+4)^3}
& \text{ if $n$ is even.}
\end{cases}\]
\end{corollary}

Based on experimental evidence gathered using Sage \cite{Sage}, we make
the following conjecture:

\begin{conjecture} \label{all-d-n}
The graph $\SR(d,n)$ is integral for all $d$ and $n$.
\end{conjecture}

We discuss the general case in Section~\ref{arbitrary-d}.
The construction of hexagon vectors generalizes as follows: for each
permutohedron whose vertices are lattice points in $n\Delta^{d-1}$,
its signed characteristic vector is an eigenvector of eigenvalue
$-\binom{d}{2}$ (Proposition~\ref{permutohedron-vector}).  This is in
fact the smallest eigenvalue of $\SR(d,n)$ when $n\geq\binom{d}{2}$.
Moreover, these eigenvectors are linearly independent and, for
fixed~$d$, account for ``almost all'' of the spectrum as $n\to\infty$,
in the sense that
\[\lim_{n\to\infty}\frac{\dim\text{(span of permutohedron eigenvectors)}}{|V(d,n)|}=1.\]

When $n<\binom{d}{2}$, the simplex $n\Delta^{d-1}$ is too small to
contain any lattice permutohedra.  On the other hand, the signed
characteristic vectors of \emph{partial permutohedra} (i.e.,
intersections of lattice permutohedra with $\SR(d,n)$) are
eigenvectors with eigenvalue~$-n$.  Experimental evidence indicates
that this is in fact the smallest eigenvalue of $A(d,n)$, and that
these partial permutohedra form a basis for the corresponding
eigenspace.  Unexpectedly, its dimension appears to be the
\emph{Mahonian number} $M(d,n)$ of permutations in $\Sym_d$ with
exactly $n$ inversions (sequence \#A008302 in Sloane \cite{OEIS}).
In Section~\ref{Ferrers-section}, we construct a family of eigenvectors
by placing rooks (ordinary rooks, not simplicial rooks!)
on Ferrers boards.

\section{Proof of the Main Theorem} \label{proof-main}

We begin by reviewing some basic algebraic graph theory; for a general reference,
see, e.g., \cite{GodRoy}.  Let $G=(V,E)$ be a simple undirected graph with
$N$ vertices.  The \emph{adjacency matrix} $A(G)$ is the $N\x N$ matrix whose $(i,j)$ entry
is 1 if vertices~$i$ and $j$ are adjacent, 0 otherwise.  The \emph{Laplacian matrix} is
$L(G)=D(G)-A(G)$, where $D(G)$ is the diagonal matrix of vertex degrees.  These are both real symmetric
matrices, so they are diagonalizable, with real eigenvalues, and eigenspaces with different eigenvalues are orthogonal
\cite[\S8.4]{GodRoy}.

\begin{proposition} \label{regular}
The graph $\SR(d,n)$ has $\binom{n+d-1}{d-1}$ vertices and
is regular of degree $(d-1)n$.  In particular, its adjacency and Laplacian matrices have the same eigenvectors.
\end{proposition}
\begin{proof}
Counting vertices is the classic ``stars-and-bars'' problem (with $n$ stars and $d-1$ bars).
For each $x\in V(d,n)$ and each pair of coordinates $i,j$, there are $x_i+x_j$ other vertices
that agree with $x$ in all coordinates but $i$ and $j$.  Therefore, the degree of $x$ is
$\sum_{1\leq i<j\leq n} (x_i+x_j) = (d-1)\sum_{i=1}^n x_i = (d-1)n$.
\end{proof}

The matrices $A(d,n)$ and $L(d,n)$ act on the vector space
$\Rr^N$ with standard basis $\{\e_{ijk} \st
(i,j,k)\in V(d,n)\}$.  We will sometimes consider the standard basis vectors
as ordered lexicographically, for the purpose of showing that a collection
of vectors is linearly independent.

In the rest of this section, we focus exclusively on the case $d=3$, and regard $n$ as fixed.  We fix $N:=\binom{n+2}{2}$,
the number of vertices of $\SR(3,n)$, and abbreviate $A=A(3,n)$.

\subsection{Basic linear algebra calculations}

Define
\begin{align*}
\X_i&:=\sum_{j+k=n-i} \e_{ijk}, & \J &:= \sum_{i+j+k=n} \e_{ijk}, \\
\Y_j&:=\sum_{i+k=n-j} \e_{ijk}, & \BB_n &:= \{\X_i,\Y_i,\Z_i \st 0\leq i\leq n\},\\
\Z_k&:=\sum_{i+j=n-k} \e_{ijk}, & \BB'_n &:= \{\X_i,\Y_i,\Z_i \st 0\leq i\leq n-1\}.
\end{align*}
The vectors $\X_i,\Y_j,\Z_k$ are the characteristic vectors of lattice lines in
$n\Delta^2$; see Figure~\ref{hxyz-figure}.  Note that the symmetric group $\Sym_3$
acts on $\SR(3,n)$ (hence on each of its eigenspaces)
by permuting the coordinates of vertices.

\begin{lemma} \label{XYZ-relations}
We have
\[\J = \sum_{i=0}^n \X_i = \sum_{i=0}^n \Y_i = \sum_{i=0}^n \Z_i
\qqandqq
n\J = \sum_{i=0}^n i(\X_i+\Y_i+\Z_i).
\]
\end{lemma}
\begin{proof}
The first assertion is immediate.  For the second, when we expand the sum in terms of the $\e_{ijk}$,
the coefficient on each $\e_{ijk}$ is $i+j+k=n$.
\end{proof}

\begin{proposition}
For every $i,j,k$, we have
\begin{subequations}
\begin{align}
A \e_{ijk} &= \X_i+\Y_j+\Z_k-3\e_{ijk},\label{Aeijk}\\
A \J &= 2n\J, \label{AJ}\\
A\X_i &= (n-i-2) \X_i + \sum_{j=0}^{n-i}\Big[\Y_j+\Z_j\Big], \label{AX}\\
A\Y_i &= (n-i-2) \Y_i + \sum_{j=0}^{n-i}\Big[\X_j+\Z_j\Big], \label{AY}\\
A\Z_i &= (n-i-2) \Z_i + \sum_{j=0}^{n-i}\Big[\X_j+\Y_j\Big]. \label{AZ}
\end{align}
\end{subequations}
\end{proposition}
\begin{proof}
Formula \eqref{Aeijk} is immediate from the definition of $A$, and \eqref{AJ} follows
because $\SR(3,n)$ is ($2n$)-regular.  For \eqref{AX}, we have
\begin{align*}
A\X_i &= \sum_{j+k=n-i} A\e_{i,j,k} = \sum_{j+k=n-i} [X_i+Y_j+Z_k-3\e_{i,j,k}]\\
&= (n-i+1) X_i - 3\sum_{j+k=n-i}\e_{i,j,k} + \sum_{j+k=n-i}[Y_j+Z_k]\\
&= (n-i-2) X_i + \sum_{j=0}^{n-i}[Y_j+Z_j]
\end{align*}
and \eqref{AY} and \eqref{AZ} are proved similarly.
\end{proof}

For future use, we also record (without proof) some elementary summation formulas.
\begin{lemma} \label{summations-for-appalling-one}
The following summations hold:
\begin{align*}
\sum_{i=k+1}^{n-k-1} \big[ 4i-2n \big] &= 0, & \sum_{i=k+1}^{n-k-1} \big[ 4i-2k-2-n \big] &= (n-2k-1)(n-2k-2),\\
\sum_{i=k+1}^{n-j} \big[ 4i-2n \big] &= 2(n-j-k)(k-j+1), & \sum_{i=k+1}^{n-j} \big[ 4i-2k-2-n \big] &= (n-2j)(n-k-j).
\end{align*}
\end{lemma}
\begin{lemma} \label{summations-for-appalling-two}
The following summations hold:
\begin{align*}
\sum_{i=k}^{n-k} \big[ 4i-2n\big] &= 0, &
\sum_{i=k}^{n-k} \big[ 4i-3n+2k-2\big] &= -(n-2k+1)(n-2k+2), \\
\sum_{i=k}^{n-j} \big[ 4i-2n\big]  &= 2 (j - k) (-n + j + k - 1),&
\sum_{i=k}^{n-j} \big[ 4i-3n+2k-2\big] &= (2 j + 2 - 4 k + n) (-n + j + k - 1).
\end{align*}
\end{lemma}

Having completed these preliminaries, we now construct the eigenvectors
of $SR(3,n)$.

\subsection{Hexagon vectors}

Let $(a,b,c)\in V(3,n)$ with $a,b,c>0$.  The corresponding ``hexagon vector''
is defined as
\[
\H_{a,b,c}:=\e_{a-1,b,c+1}-\e_{a,b-1,c+1}+\e_{a+1,b-1,c}-\e_{a+1,b,c-1}+\e_{a,b+1,c-1}-\e_{a-1,b+1,c}.
\]
Geometrically, this is the characteristic vector, with alternating signs,
of a regular lattice hexagon centered at the lattice point $(a,b,c)$
in the interior of $n\Delta^2$ (see Figure~\ref{hxyz-figure}).

\begin{figure}[ht] 
\includefigure{4.9in}{1.44in}{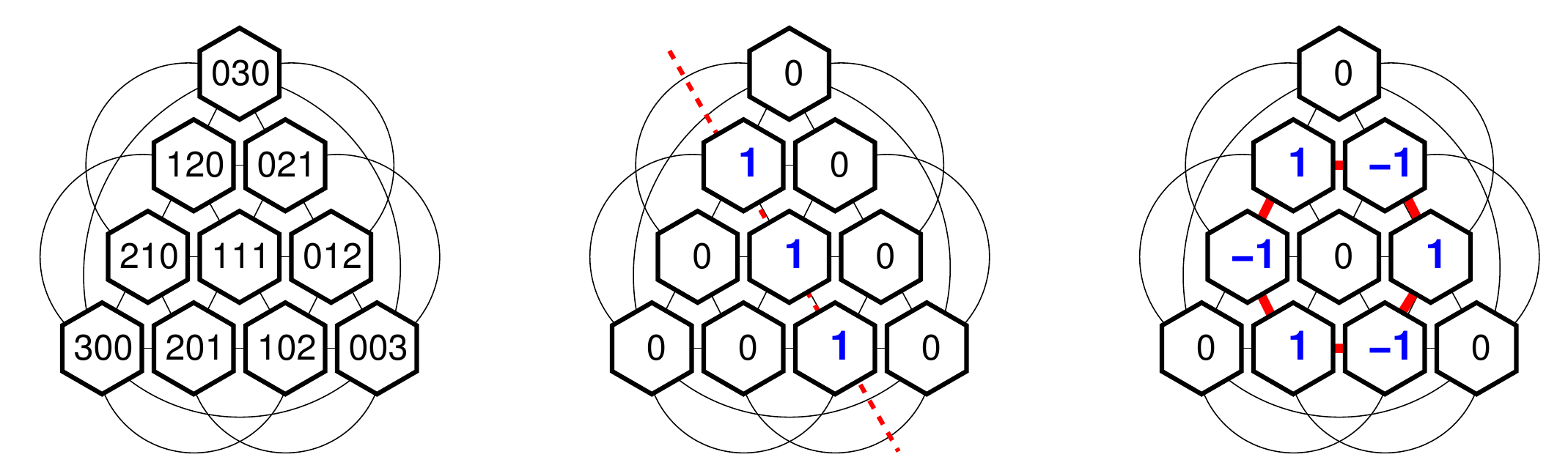}
\caption{\label{hxyz-figure} (left) The graph $\SR(3,3)$.
(center) The vector $\X_1$ and the lattice line it supports.
(right) $\H_{1,1,1}$.}
\end{figure}

\begin{proposition} \label{hex-prop}
The vectors $\{\H_{a,b,c}\st (a,b,c)\in V(d,n),\; a,b,c>0\}$ are linearly independent,
and each one is an eigenvector of $A$ with eigenvalue $-3$.
\end{proposition}
\begin{proof}
The equality $A\H_{a,b,c}=-3\H_{a,b,c}$ is straightforward from~\eqref{Aeijk}.
The lexicographic leading term of $\H_{a,b,c}$ is $\e_{a-1,b,c+1}$, which is different
for each $(a,b,c)$, implying linear independence.
\end{proof}

\begin{proposition} \label{basis-lemma}
Let $n\geq 1$ and let $\HH_n=\{\H_{a,b,c} \st 0<a,b,c<n\}$.
Then the spaces $\Rr\HH_n$ and $\Rr\BB_n$ spanned by $\HH_n$ and $\BB_n$ are orthogonal complements in $\Rr^N$.  In particular, $\dim\Rr\BB_n=\binom{n+2}{2}-\binom{n-1}{2} = 3n$, and the set $\BB_n'$ is a basis for $\Rr\BB_n$
(and all linear relations on the $\X_i,\Y_i,\Z_i$ are generated by those of Lemma~\ref{XYZ-relations}).
\end{proposition}

\begin{proof}
The scalar product $\H_{a,b,c}\cdot\X_i$ is clearly zero if the two vectors have disjoint supports
(i.e., $i\not\in\{a-1,a,a+1\}$) and is $-1+1=0$ otherwise (geometrically, this
corresponds to the statement that any two adjacent vertices in the hexagon occur
with opposite signs in $\H_{a,b,c}$; see Figure~\ref{hxyz-figure}).
Therefore $\Rr\HH_n$ and $\Rr\BB_n$ are orthogonal subspaces of~$\Rr^N$, and $\dim\Rr\BB_n\leq 3n$.
For the opposite inequality, we induct on $n$.  In the base case $n=1$, the vectors $X_0,Y_0,Z_0$
form a basis of $\Rr^3$.  For larger $n$, let $M_n$ be the matrix with columns $X_n,Y_n,Z_n,\dots,X_0,Y_0,Z_0$
and rows ordered lexicographically, and let $\tilde M_n$ be $M_n$ with the columns reordered as
$$X_0,\; Y_n,\; Z_n,\ \ X_n,\; Y_{n-1},\; Z_{n-1},\ \ \dots,\ \ X_1,\; Y_0,\; Z_0.$$
For example,
\[
\tilde M_3=\begin{array}{c|ccc|ccccccccc}
& X_0 & Y_3 & Z_3 & X_3 & Y_2 & Z_2 & X_2 & Y_1 & Z_1 & X_1 & Y_0 & Z_0 \\\hline
003 & 1 & 0 & 1 & 0 & 0 & 0 & 0 & 0 & 0 & 0 & 1 & 0\\
012 & 1 & 0 & 0 & 0 & 0 & 1 & 0 & 1 & 0 & 0 & 0 & 0\\
021 & 1 & 0 & 0 & 0 & 1 & 0 & 0 & 0 & 1 & 0 & 0 & 0\\
030 & 1 & 1 & 0 & 0 & 0 & 0 & 0 & 0 & 0 & 0 & 0 & 1\\\hline
102 & 0 & 0 & 0 & 0 & 0 & 1 & 0 & 0 & 0 & 1 & 1 & 0\\
111 & 0 & 0 & 0 & 0 & 0 & 0 & 0 & 1 & 1 & 1 & 0 & 0\\
120 & 0 & 0 & 0 & 0 & 1 & 0 & 0 & 0 & 0 & 1 & 0 & 1\\
201 & 0 & 0 & 0 & 0 & 0 & 0 & 1 & 0 & 1 & 0 & 1 & 0\\
210 & 0 & 0 & 0 & 0 & 0 & 0 & 1 & 1 & 0 & 0 & 0 & 1\\
300 & 0 & 0 & 0 & 1 & 0 & 0 & 0 & 0 & 0 & 0 & 1 & 1\\
\end{array}\]

If $a>0$, then the entries of $M_n$ in row $(a,b,c)$ and columns
$X_i,Y_i,Z_i$ equal the entries of $M_{n-1}$ in row $(a-1,b,c)$
and columns $X_{i-1},Y_i,Z_i$ respectively.  Hence~$\tilde
M_n$ has the block form $\displaystyle\left[\begin{array} {c|c} U & *\\\hline 0 & M_{n-1}\end{array} \right]$,
where the entries of $*$ are irrelevant and
\[
U = \begin{bmatrix} 1 & 0 & 1\\ 1 & 0 & 0\\ \vdots&\vdots&\vdots\\ 1 & 0 & 0&\\ 1 & 1 & 0\end{bmatrix}.\]
Since $\rank U=3$, it follows by induction that $\rank M_n \geq \rank M_{n-1}+3 = 3n$.
Using Lemma~\ref{XYZ-relations}, one can solve for each of
$\X_n$, $\Y_n$, and $\Z_n$ as linear combinations of the vectors in $\BB'_n$.  It follows that $\BB'_n$ is a basis,
and that the linear relations of Lemma~\ref{XYZ-relations} generate all linear relations
on the vectors $\{\X_i,\Y_i,\Z_i\}$.
\end{proof}

\subsection{Non-Hexagon Eigenvectors}

We now determine the other eigenspaces of~$A$.  The vector $\J$ spans
an eigenspace of dimension 1; in addition, we will show that there is
one eigenspace of dimension~2 (Prop.~\ref{dim-two-eigenspace}) and two
families of eigenspaces of dimension~3
(Props.~\ref{appalling-calculation-one}
and~\ref{appalling-calculation-two}).  Together with the hexagon
vectors, these form a complete decomposition of $\Rr^N$ into
eigenspaces of~$A$.  Throughout, let $\sigma$ and $\rho$ denote the
permutations $(1\ 2\ 3)$ and $(1\ 2)$ (written in cycle notation),
respectively, so that
$$
\sigma(\X_i)=\Y_i,\ \ \sigma(\Y_j)=\Z_j,\ \ \sigma(\Z_k)=\X_k,\ \ 
\rho(\X_i)=\Y_i,\ \ \rho(\Y_j)=\X_j,\ \ \rho(\Z_k)=\Z_k.
$$

\begin{proposition} \label{dim-two-eigenspace}
Let $n\geq1$ and $k=\lfloor n/2\rfloor$.  Then
$$\R := \X_k-\Y_k-\X_{k+1}+\Y_{k+1}$$
is a nonzero eigenvector of $A$ with eigenvalue $n-k-3=(n-6)/2$ if $n$ is even,
or $n-k-2=(n-3)/2$ if $n$ is odd.  Moreover, the $\Sym_3$-orbit of $\R$ has dimension 2.
\end{proposition}

\begin{proof}
By \eqref{AX}\dots\eqref{AZ},
\begin{align*}
A\R  &= (n-k-2)(\X_k - \Y_k) + \sum_{j=0}^{n-k}\Big[\Y_j-\X_j\Big] +(n-k-3)(\Y_{k+1}-\X_{k+1}) + \sum_{j=0}^{n-k-1}\Big[\X_j-\Y_j\Big]\\
 &= (n-k-2)(\X_k - \Y_k) + (\Y_{n-k}-\X_{n-k}) +(n-k-3)(\Y_{k+1}-\X_{k+1}) \\\\
 &= \begin{cases}
       (n-k-2)(\X_k - \Y_k) + (\Y_{k}-\X_{k}) +(n-k-3)(\Y_{k+1}-\X_{k+1}) & \text{ if $n$ is even},\\
       (n-k-2)(\X_k - \Y_k) + (\Y_{k+1}-\X_{k+1}) +(n-k-3)(\Y_{k+1}-\X_{k+1}) & \text{ if $n$ is odd}
       \end{cases}\\\\
 &= \begin{cases}
       (n-k-3)(\X_k - \Y_k)+(n-k-3)(\Y_{k+1}-\X_{k+1}) & \text{ if $n$ is even},\\
       (n-k-2)(\X_k - \Y_k) +(n-k-2)(\Y_{k+1}-\X_{k+1}) & \text{ if $n$ is odd}
       \end{cases}\\\\
 &= \begin{cases}
       (n-k-3) \R & \text{ if $n$ is even},\\
       (n-k-2) \R & \text{ if $n$ is odd}
       \end{cases}
\end{align*}
as desired.  The vectors $\R$ and $\sigma(\R)=\Y_k-\Z_k-\Y_{k+1}+\Z_{k+1}$
are linearly independent; on the other hand, $\rho(\R)=\R$ and $\R+\sigma(\R)+\sigma^2(\R)=0$,
so the $\Sym_3$-orbit of $\R$ has dimension 2.
\end{proof}

\begin{proposition} \label{appalling-calculation-one}
For all integers $k$ with $0\leq k\leq \lfloor\frac{n-3}{2}\rfloor$, the vector
\[
\P_k := - (n-2k-1)(n-2k-2) \Z_{n-k} + \sum_{i=k+1}^{n-k-1} \Big[ 2(i-k-1)\Z_i+(2i-n)(\X_i+\Y_i) \Big]
\]
is a nonzero eigenvector of $A$ with eigenvalue $k-2$.
Moreover, the $\Sym_3$-orbit of $\P_k$ has dimension 3.
\end{proposition}

\begin{proof}
The upper bound on $k$ is equivalent to $n-2k-2>0$, so the coefficient of $\Z_{n-k}$ in $\P_k$ is nonzero,
so $\P_k\neq 0$.  By \eqref{AX}\dots\eqref{AZ}, we have
\begin{align*}
A\P_k &= - (n-2k-1)(n-2k-2)\left( (k-2) \Z_{n-k} + \sum_{i=0}^k\Big[\X_i+\Y_i\Big] \right)
 \\&\qquad+ \sum_{i=k+1}^{n-k-1}\left[ 2(i-k-1)\left((n-i-2) \Z_i + \sum_{j=0}^{n-i}\Big[\X_j+\Y_j\Big]\right)\right.
 \\&\qquad\left.+(2i-n)\left((n-i-2)(\X_i+\Y_i) + \sum_{j=0}^{n-i}\Big[\X_j+\Y_j+2\Z_j\Big]\right) \right]
\\
&= - (n-2k-1)(n-2k-2)(k-2) \Z_{n-k}  - (n-2k-1)(n-2k-2)\sum_{i=0}^k \Big[ \X_i+\Y_i \Big]
 \\&\qquad+ \sum_{i=k+1}^{n-k-1} \Big[ (2i-n)(n-i-2)(\X_i+\Y_i) + 2(i-k-1)(n-i-2) \Z_i \Big]
 \\&\qquad+ \sum_{i=k+1}^{n-k-1} \sum_{j=0}^{n-i} \Big[ (4i-2k-2-n) (\X_j+\Y_j) + (4i-2n)\Z_j \Big].
\end{align*}
Interchanging the order of summation in the double sum gives
\begin{align*}
A\P_k&=- (n-2k-1)(n-2k-2)(k-2) \Z_{n-k}\\
&\qquad- (n-2k-1)(n-2k-2)\sum_{i=0}^k\Big[\X_i+\Y_i\Big]
 \\&\qquad+ \sum_{i=k+1}^{n-k-1} \Big[ (2i-n)(n-i-2)(\X_i+\Y_i) + 2(i-k-1)(n-i-2) \Z_i \Big]
 \\&\qquad+ \sum_{j=0}^{k} \sum_{i=k+1}^{n-k-1} \Big[ (4i-2k-2-n) (\X_j+\Y_j) + (4i-2n)\Z_j \Big]
 \\&\qquad+ \sum_{j=k+1}^{n-k-1} \sum_{i=k+1}^{n-j} \Big[ (4i-2k-2-n) (\X_j+\Y_j) + (4i-2n)\Z_j \Big]
 \end{align*}
Applying the summation formulas of  Lemma~\ref{summations-for-appalling-one} gives
\begin{align*}
A\P_k&= - (n-2k-1)(n-2k-2)(k-2) \Z_{n-k}  - (n-2k-1)(n-2k-2)\sum_{i=0}^k\Big[\X_j+\Y_j\Big]
 \\&\qquad+ \sum_{i=k+1}^{n-k-1} \Big[(2i-n)(n-i-2)(\X_i+\Y_i) + 2(i-k-1)(n-i-2) \Z_i\Big]
 \\&\qquad+ \sum_{j=0}^k \Big[(n-2k-1)(n-2k-2) (\X_j+\Y_j)\Big]
 \\&\qquad+ \sum_{j=k+1}^{n-k-1} \Big[(2j-n)(k+j-n) (\X_j+\Y_j) + 2(j-n+k)(j-1-k)\Z_j\Big]
\\
&= - (n-2k-1)(n-2k-2)(k-2) \Z_{n-k} \\&\qquad + \sum_{i=k+1}^{n-k-1} \Big[ (2i-n)(k-2)(\X_i+\Y_i) + 2(i-k-1)(k-2) \Z_i\Big]\\
&= (k-2) \left( -(n-2k-1)(n-2k-2) \Z_{n-k} + \sum_{i=k+1}^{n-k-1} \Big[ (2i-n)(\X_i+\Y_i) + 2(i-k-1) \Z_i \Big] \right)\\
&= (k-2)\P_k
\end{align*}
verifying that $\P_k$ is an eigenvector, as desired.

Consider the vectors $\P_k$, $\sigma(\P_k)$, $\sigma^2(\P_k)$ as elements of the
vector space $\Rr\BB_n$, expanded in terms of the basis $\BB_n'$
(see Prop.~\ref{basis-lemma}).  In these expansions,
the basis vectors $\Z_{n-k}$, $\X_{n-k}$, $\Y_{n-k}$ occur
with nonzero coefficients only in $\P_k$, $\sigma(\P_k)$, $\sigma^2(\P_k)$ respectively.  This shows that these three vectors are linearly independent.  On the other hand,
$\rho(\P_k)=\P_k$, so the $\Sym_3$-orbit of $\P_k$ has dimension~3.
\end{proof}

\begin{subequations}
Define
\begin{align}
\Q_k &:= (n-2k+1)(n-2k+2) \Z_k\notag\\
&\qquad + \sum_{j=k}^{n-k}\Big[ (2j-n)(\X_j+\Y_j) - 2(n-j-k+1) \Z_j \Big] \label{QK-one}\\
&= (n-2k+1)(n-2k) \Z_k  +(2k-n)(\X_k+\Y_k)\notag\\
&\qquad + \sum_{j=k+1}^{n-k}\Big[ (2j-n)(\X_j+\Y_j) - 2(n-j-k+1) \Z_j \Big].\label{QK-two}
\end{align}
\end{subequations}
Both of these expressions for $\Q_n$ will be useful in what follows.

\begin{proposition} \label{appalling-calculation-two}
For all integers $k$ with $0\leq k\leq \lfloor\frac{n-2}{2}\rfloor$, the vector $\Q_k$ is a nonzero eigenvector of $A$ with eigenvalue $n-k-2$.
Moreover, the $\Sym_3$-orbit of $\Q_k$ has dimension 3.
\end{proposition}

\begin{proof}
The statement is vacuously true if $n<2$.
By \eqref{QK-one}, the coefficient of  $\Z_k$ in~$\Q_k$ is $(n-2k+1)(n-2k)$.  Provided that $k\leq\lfloor\frac{n-1}{2}\rfloor$,
we have $n>2k$, so this coefficient is nonzero, as is the vector~$\Q_k$.
Applying \eqref{AX}\dots\eqref{AZ}, we have
\begin{align*}
A\Q_k
&= (n-2k+1)(n-2k+2) \left((n-k-2) \Z_k + \sum_{j=0}^{n-k}\Big[\X_j+\Y_j\Big] \right)\\
&\qquad  + \sum_{i=k}^{n-k}\left[ (2i-n)\left((n-i-2)(\X_i+\Y_i) + \sum_{j=0}^{n-i}\Big[\X_j+\Y_j+2\Z_j\Big]\right) \right.\\
&\qquad \left. - 2(n-i-k+1) \left((n-i-2) \Z_i + \sum_{j=0}^{n-i}\Big[\X_j+\Y_j\Big]\right) \right]\\
&= (n-2k+1)(n-2k+2)(n-k-2)\Z_k + (n-2k+1)(n-2k+2)\sum_{j=0}^{n-k}\Big[\X_j+\Y_j\Big]\\
&\qquad + \sum_{i=k}^{n-k} \Big[ (2i-n)(n-i-2)(\X_i+\Y_i) - 2(n-i-k+1)(n-i-2)\Z_i \Big]\\
&\qquad + \sum_{i=k}^{n-k}\sum_{j=0}^{n-i} \Big[ (2i-n)(\X_j+\Y_j+2\Z_j) -2(n-i-k+1)(\X_j+\Y_j) \Big]
\end{align*}
Interchanging the order of summation in the double sum gives
\begin{align*}
A\Q_k
&= (n-2k+1)(n-2k+2)(n-k-2)\Z_k + (n-2k+1)(n-2k+2)\sum_{j=0}^{n-k}\Big[\X_j+\Y_j\Big]
\\&\qquad + \sum_{j=k}^{n-k} \Big[ (2j-n)(n-j-2)(\X_j+\Y_j) - 2(n-j-k+1)(n-j-2)\Z_j \Big]
\\&\qquad + \sum_{j=0}^{k-1}\sum_{i=k}^{n-k} \Big[ (4i-2n)\Z_j + (4i-3n+2k-2)(\X_j+\Y_j) \Big]
\\&\qquad + \sum_{j=k}^{n-k}\sum_{i=k}^{n-j} \Big[ (4i-2n)\Z_j + (4i-3n+2k-2)(\X_j+\Y_j) \Big]
\end{align*}
Applying the summation formulas of Lemma~\ref{summations-for-appalling-two} gives
\begin{align*}
A\Q_k &= (n-2k+1)(n-2k+2)(n-k-2)\Z_k + (n-2k+1)(n-2k+2)\sum_{j=0}^{n-k}\Big[\X_j+\Y_j\Big]\\
&\qquad + \sum_{j=k}^{n-k} \Big[ (2j-n)(n-j-2)(\X_j+\Y_j) - 2(n-j-k+1)(n-j-2)\Z_j \Big]\\
&\qquad - \sum_{j=0}^{k-1} \Big[ (n-2k+1)(n-2k+2)(\X_j+\Y_j) \Big]\\
&\qquad + \sum_{j=k}^{n-k} \Big[ 2 (j - k) (-n + j + k - 1)\Z_j + (2 j + 2 - 4 k + n) (-n + j + k - 1)(\X_j+\Y_j) \Big]\\
&= (n-2k+1)(n-2k+2)(n-k-2)\Z_k\\
&\qquad+ \sum_{j=k}^{n-k} \Big[ (n-k-2)(2j-n)(\X_j+\Y_j) - 2(n-j-k+1)(n-k-2)\Z_j \Big]\\
&= (n-k-2) \left( (n-2k+1)(n-2k+2)\Z_k + \sum_{j=k}^{n-k} \Big[ (2j-n)(\X_j+\Y_j) - 2(n-j-k+1)\Z_j \Big] \right) \\
&= (n-k-2) \Q_k
\end{align*}
as desired.

We now show that the $\Sym_3$-orbit of $\Q_k$ has dimension~3.
Since $\rho(\Q_k)=\Q_k$, the orbit is spanned by the three
vectors $\Q_k$, $\sigma(\Q_k)$, $\sigma^2(\Q_k)$.
We consider two cases: $k=0$ and $k>0$.

First, if $k=0$, then the expression~\eqref{QK-one} for $\Q_0$ becomes
(using Lemma~\ref{XYZ-relations})
\begin{align*}
\Q_0 &=
 (n+1)(n+2) \Z_0  + \sum_{j=0}^{n}\Big[ (2j-n)(\X_j+\Y_j) - 2(n-j+1) \Z_j \Big]\\
&=
 (n+1)(n+2) \Z_0  - \sum_{j=0}^{n}\Big[ n(\X_j+\Y_j) + (2n+2) \Z_j \Big]
+2\sum_{j=0}^n j\Big[\X_j+\Y_j+\Z_j\Big]\\
&= (n+1)(n+2) \Z_0  - (4n+2) \J
 +2n\J ~=~  (n^2+3n+2) \Z_0  - (2n+2) \J \\
&= \sum_{i+j=n} (n^2+n) \e_{ij0} + \sum_{i,j,k \st k\neq 0} (-2n-2) \e_{ijk}.
\end{align*}
Accordingly we have
\begin{align*}
\sigma(\Q_0) &= \sum_{j+k=n} (n^2+n) \e_{0jk} + \sum_{i,j,k \st i\neq 0} (-2n-2) \e_{ijk},\\
\sigma^2(\Q_0) &= \sum_{i+k=n} (n^2+n) \e_{i0k} + \sum_{i,j,k \st j\neq 0} (-2n-2) \e_{ijk}.
\end{align*}
Consider the $N\x 3$ matrix with columns $\Q_0,\sigma(\Q_0),\sigma^2(\Q_0)$.  By the previous calculation, the $3\x3$ minor in rows $\e_{n00},\e_{0n0},\e_{00n}$ is
$$\left\vert\begin{array}{ccc}
n^2+n & -2n-2 & n^2+n\\
n^2+n & n^2+n & -2n-2\\
-2n-2 & n^2+n & n^2+n\end{array}\right\vert
=-2(n+1)^3(n+2)^2(n-1)$$
which is nonzero (recall that $n\geq 2$, otherwise the proposition is vacuously true).

On the other hand, if $0<k\leq\lfloor(n-2)/2\rfloor$,
then~\eqref{QK-two} expresses
$\Q_k,\sigma(\Q_k),\sigma^2(\Q_k)$ as column vectors in the
basis $\BB'_n$.  Let $a=2k-n$ and $b=(n-2k)(n-2k+1)$; then
the $3\x3$ minor in rows $\X_k,\Y_k,\Z_k$ is
$$\left\vert\begin{array}{ccc}
a & a & b\\
a & b & a\\
b & a & a\end{array}\right\vert
=(2k-n)^3(n-2k-1)(n-2k+2)^2
$$
which is nonzero because the assumption
$k\leq\lfloor(n-2)/2\rfloor$ implies $n\geq 2k+2$.
\end{proof}

To sum up the results of Section~\ref{proof-main}, we have constructed an explicit decomposition of $\Rr^N$ into eigenspaces of $A(3,n)$ (equivalently, $L(3,n)$).  The eigenvectors are the hexagon vectors $\H_{a,b,c}$ and the special vectors ${\bf J}$, ${\bf R}$, ${\bf P}_k$ and  ${\bf Q}_k$ and their $\Sym_3$-orbits.

\section{Simplicial rook graphs in arbitrary dimension} \label{arbitrary-d}

We now consider the graph $\SR(d,n)$ for arbitrary $d$ and $n$,
with adjacency matrix $A=A(d,n)$.  Recall that $\SR(d,n)$ has $N:=\binom{n+d-1}{d-1}$
vertices and is regular of degree $(d-1)n$.
If two vertices $a=(a_1,\dots,a_d)$, $b=(b_1,\dots,b_d)\in V(d,n)$ differ only in their $i^{th}$ and $j^{th}$ positions
(and are therefore adjacent), we write $a\simij b$.

Let $\Sym_d$ be the symmetric group of order~$d$,
and let $\Alt_d\subset\Sym_d$ be the alternating subgroup.  Let $\sgn$ be the sign function
\[\sgn(\sigma)=\begin{cases} 1& \text{ for } \sigma\in\Alt_d,\\ -1& \text{ for } \sigma\not\in\Alt_d.\end{cases}\]
Let $\tau_{ij}\in\Sym_d$ denote the transposition of $i$ and $j$.  Note that $\Sym_d=\Alt_d\cup\Alt_d\tau_{ij}$
for each $i,j$.

In analogy to the vectors $\X,\Y,\Z$ used in the $d=3$ case, define
\begin{equation} \label{higher-X}
\X_\alpha^{(i,j)} = \e_\alpha + \sum_{\beta:\ \beta\simij\alpha} \e_\beta.
\end{equation}
That is, $\X_\alpha^{(i,j)}$ is the characteristic vector of the lattice line through
$\alpha$ in direction $\e_i-\e_j$.  In particular, if $\alpha\simij\beta$, then $\X_\alpha^{(i,j)}=\X_\beta^{(i,j)}$.  Moreover,
the column of $A$ indexed by $\alpha$ is
\begin{equation} \label{higher-Ae}
A \e_\alpha = -\binom{d}{2}\e_\alpha + \sum_{1\leq i<j\leq d} \X_\alpha^{(i,j)}.
\end{equation}
since $e_\alpha$ itself appears in each summand $\X_\alpha^{(i,j)}$.

\subsection{Permutohedron vectors}
\label{permuto-vector-subsection}

We now generalize the construction of hexagon vectors to arbitrary
dimension.  The idea is that for each point $p$ in the interior of
$n\Delta^{d-1}$ and sufficiently far away from its boundary, there is
a lattice permutohedron centered at $p$, all of whose points are
vertices of $\SR(d,n)$ (see Figure~\ref{perm-figure}), and the
signed characteristic vector of this permutohedron is an eigenvector
of~$A(d,n)$.

\begin{figure}[ht]
\includefigure{2.5in}{2.5in}{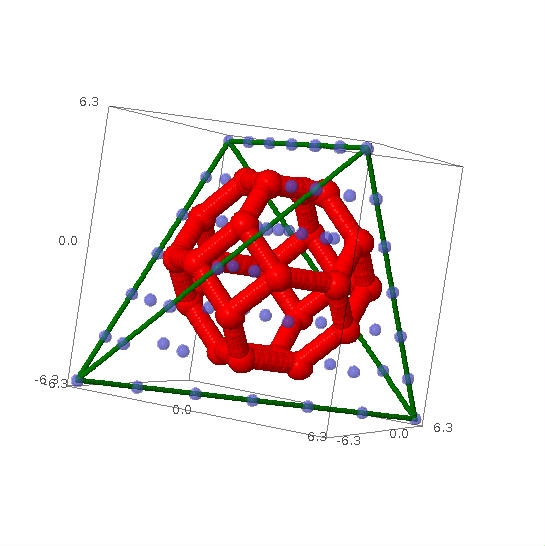}
\caption{\label{perm-figure} A permutohedron vector ($n=6$, $d=4$).}
\end{figure}

\begin{proposition} \label{permutohedron-vector}
Let $p,w\in\Rr^N$ be vectors such that $\{p+\sigma(w) \st \sigma\in\Sym_d\}$ are distinct vertices of $\SR(d,n)$.
(In particular, the entries of $w$ must all be different.)  Define
$$\H_{p,w} = \sum_{\sigma\in\Sym_d} \sgn(\sigma) \e_{p+\sigma(w)}.$$
Then $\H_{p,w}$ is an eigenvector of $A$ with eigenvalue $-\binom{d}{2}$.  Moreover,
for a fixed $w$, the collection of all such eigenvectors $\H_{p,w}$ is linearly independent.
\end{proposition}

\begin{proof} By linearity and \eqref{higher-Ae}, we have
\begin{align*}
A \H_{p,w}
&= \sum_{\sigma\in\Sym_d} \sgn(\sigma) \left( -\binom{d}{2}\e_{p+\sigma(w)}+\sum_{1\leq i<j\leq d} \X^{(i,j)}_{p+\sigma(w)} \right) \\
&= -\binom{d}{2} \H_{p,w} + \sum_{\sigma\in\Sym_d} \sgn(\sigma) \sum_{1\leq i<j\leq d} \X^{(i,j)}_{p+\sigma(w)}\\
&= -\binom{d}{2} \H_{p,w} + \sum_{1\leq i<j\leq d} \sum_{\sigma\in\Sym_d} \sgn(\sigma) \X^{(i,j)}_{p+\sigma(w)}\\
&= -\binom{d}{2} \H_{p,w} + \sum_{1\leq i<j\leq d} \sum_{\sigma\in\Alt_d} \big[
\sgn(\sigma)\X^{(i,j)}_{p+\sigma(w)} + \sgn(\sigma\tau_{ij})\X^{(i,j)}_{p+\sigma\tau_{ij}(w)} \big]\\
&= -\binom{d}{2} \H_{p,w}.
\end{align*}
(The summand vanishes because $\sgn(\sigma)=-\sgn(\sigma\tau_{ij})$ and because
changing $\alpha_i$ and $\alpha_j$ does not change $\X^{(i,j)}_{\alpha}$.)
For linear independence, it suffices to observe that the lexicographic leading term of $\H_{p,w}$ is $\e_{p+\tilde w}$,
where $\tilde w$ denotes the unique increasing permutation of $w$, and that these leading terms
are different for different~$p$.
\end{proof}

This result says that we can construct a large eigenspace by fitting many congruent permutohedra into
the dilated simplex.  Depending on the parity of $d$, the centers of these permutohedra will be points in
$\Zz^d$ or $(\Zz+\frac12)^d$.

Let $d$ be a positive integer.  The \emph{standard offset vector} in~$\Rr^d$ is defined
as
\begin{equation} \label{standard-offset}
\soff=\soff_d=((1-d)/2, (3-d)/2, \dots, (d-3)/2, (d-1)/2)\in\Rr^d.
\end{equation}
Note that $\soff\in\Zz^d$ if $d$ is odd, and $\soff\in(\Zz+\frac12)^d$ if $d$ is even.

\begin{proposition} \label{pack-permutohedra}
There are
$$\binom{n-\frac{(d-1)(d-2)}{2}}{d-1}$$
distinct vectors $p$ such that $\H_{p,\soff}$ is an eigenvector of $A(d,n)$
(and these eigenvectors are all linearly independent by Prop.~\ref{permutohedron-vector}).
\end{proposition}

\begin{proof}
First, suppose that $d = 2c+1$ is odd.  In order to satisfy the conditions
of Prop.~\ref{permutohedron-vector}, it suffices to choose a lattice point $p=(a_1,\dots,a_d)$ so that 
$\sum a_i=n$ and $c\leq a_i\leq n-c$ for all $i$.
Subtracting $c$ from each $a_i$ gives a bijection to
compositions of $n-cd$ with $d$ nonnegative parts and no part greater than $n-2c$
(that latter condition is extraneous for $d\geq 2$).  The number of these compositions is
$$\binom{n-cd+d-1}{d-1}=\binom{n-\frac{(d-1)(d-2)}{2}}{d-1}.$$

Second, suppose that $d = 2c$ is even.  Now it suffices to choose a point
$p=(a_1+1/2,\dots,a_d+1/2)\in(\Zz+\frac12)^d$ such that $a_1+\dots+a_d=n-c$ and, for each $i$,
$a_i+1/2+(1-d)/2\geq 0$ and $a_i+1/2+(d-1)/2\leq n$, that is,
i.e., $c-1\leq a_1\leq n-c$.  Subtracting $c-1$ from each $a_i$ gives a bijection to
compositions of $n-c-d(c-1)=n-d(d-1)/2$ with $d$ nonnegative parts,
none of which can be greater than $n-d+1$ (again, the last condition is extraneous).
The number of these compositions is 
$$\binom{n-d(d-1)/2+d-1}{d-1}=\binom{n-\frac{(d-1)(d-2)}{2}}{d-1}.$$
\end{proof}

The permutohedron vectors account for ``almost all'' of the eigenvectors in
the following sense: if $\mathcal{H}_{d,n}\subseteq\Rr^N$ be the linear span of the eigenvectors constructed in Props.~\ref{permutohedron-vector}
and~\ref{pack-permutohedra}, then for each fixed $d$, we have
\begin{equation} \label{dominates}
\lim_{n\to\infty} \frac{\dim\mathcal{H}_{d,n}}{|V(d,n)|}
=\lim_{n\to\infty} \frac{\binom{n-\frac{(d-1)(d-2)}{2}}{d-1}}{\binom{n+d-1}{d-1}}=1.
\end{equation}
On the other hand, the combinatorial structure of the remaining eigenvectors
is not clear.

The next result is a partial generalization of Proposition~\ref{basis-lemma}.

\begin{proposition} \label{HorthoX}
Every $\H_{p,\soff}$ is orthogonal to every $\X^{(i,j)}_\alpha$.
\end{proposition}
\begin{proof}
By definition we have 
\[
\X_\alpha^{(i,j)} \cdot \H_{p,\soff}
= \left( \e_\alpha + \sum_{\beta:\ \beta\simij\alpha} \e_\beta \right)\cdot
\left(\sum_{\sigma\in\Sym_d} \sgn(\sigma) \e_{p+\sigma(\soff)} \right)
= \sum_{\sigma:\ p+\sigma(\soff)\simij\alpha} \sgn(\sigma).
\]
The index set of this summation admits the fixed-point-free involution
$(\beta,\sigma)\leftrightarrow(\tau\beta,\tau\sigma)$, and $\sgn(\tau\sigma)=-\sgn(\sigma)$,
so the sum is zero.
\end{proof}

Geometrically, Proposition~\ref{HorthoX} says that if a lattice line
meets a lattice permutohedron of the form of Prop.~\ref{permutohedron-vector},
then it does so in exactly two points, whose corresponding permutations have opposite signs.

\begin{conjecture} \label{span-ortho}
The vectors $\X^{(i,j)}_\alpha$ span the orthogonal
complement of $\mathcal{H}_{d,n}$.
\end{conjecture}

This conjecture is equivalent to the statement
that every other eigenvector of $A(d,n)$ can be written as a linear
combination of the $\X^{(i,j)}_\alpha$.  For $n<\binom{d}{2}$, the conjecture
is that the $\X^{(i,j)}_\alpha$ span all of $\Rr^N$.
We have verified this statement
computationally for $d=4$ and $n\leq 11$, and for $d=5$ and $n=7,8,9$.
Part of the difficulty is that it is not clear what subset of the $\X^{(i,j)}_\alpha$ ought to form a
basis (in contrast to the case $d=3$, where $\BB'_n$ is a natural
choice of basis; see Prop.~\ref{basis-lemma}).

\subsection{The smallest eigenvalue} \label{Ferrers-section}

For a matrix $M$ with real spectrum,
let $\tau(M)$ denote its smallest eigenvalue,
and for a graph $H$, let $\tau(G)=\tau(A(G))$.
The invariant $\tau(G)$ of a graph is important in spectral graph theory; for
instance, it is related to the independence number~ \cite[Lemma 9.6.2]{GodRoy}.

\begin{proposition} \label{smallest-eigen}
Suppose that $d\geq 1$ and $n\geq\binom{d}{2}$.  Then $\tau(\SR(d,n))=-\binom{d}{2}$.
\end{proposition}

\begin{proof}
By the construction of Proposition~\ref{pack-permutohedra}, there is at least one
eigenvector with eigenvalue $-\binom{d}{2}$ when $n\geq\binom{d}{2}$.  The following argument
that $-\binom{d}{2}$ is in fact the smallest eigenvalue was suggested to the authors by Noam Elkies.
The edges of $\SR(d,n)$ in direction $(i,j)$ form a spanning subgraph $\SR(d,n)_{i,j}$ isomorphic to $K_{n+1}+K_n+K_{n-1}+\cdots+K_1$, where $+$ means disjoint union.  The eigenvalues of $K_n$ are $n-1$ and $-1$, and the spectrum of $G+H$ is the union of the spectra of $G$ and $H$, so $\tau(\SR(d,n)_{i,j})=-1$.  Since the edge set of
$\SR(d,n)$ is the disjoint union of the edge sets of the $\SR(d,n)_{i,j}$, we have
$A(d,n)=\sum_{(i,j)} A(\SR(d,n)_{i,j})$, and in general $\tau(M+N)\geq\tau(M)+\tau(N)$, so $\tau(\SR(d,n))\geq-\binom{d}{2}$ as desired.
\end{proof}

The case $n<\binom{d}{2}$ is more complicated.  Experimental evidence
indicates that the smallest eigenvalue of $\SR(d,n)$ is $-n$, and
moreover that the multiplicity of this eigenvalue equals the number
$M(d,n)$ of permutations in $\Sym_d$ with exactly $n$ inversions.  The
numbers $M(d,n)$ are well known in combinatorics as the \emph{Mahonian
numbers}, or as the coefficients of the $q$-factorial polynomials; see
\cite[sequence~\#A008302]{OEIS}.  In the rest of this
section, we construct $M(d,n)$ linearly independent eigenvectors of
eigenvalue $-n$; however, we do not know how to rule out the
possibility of additional eigenvectors of equal or smaller eigenvalue

We review some basics of rook theory; for a general reference,
see, e.g., \cite{rook-notes}.
For a sequence of positive integers $c=(c_1,\dots,c_d)$,
the \emph{skyline board} $\Sky(c)$ consists of a sequence of $d$ columns,
with the $i^{th}$ column containing $c_i$ squares.  A \emph{rook placement}
on $\Sky(c)$ consists of a choice of one square in each column.
A rook placement is \emph{proper} if all $d$ squares belong to different rows.

An \emph{inversion} of a permutation $\pi=(\pi_1,\dots,\pi_d)\in\Sym_d$ is a pair $i,j$ such that $i<j$ and $\pi_i>\pi_j$.
Let $\Sym_{d,n}$ denote the set of permutations of $[d]$ with exactly $n$ inversions.

\begin{definition}
Let $\pi\in\Sym_{d,n}$.  The \emph{inversion word} of $\pi$ is $a=a(\pi)=(a_1,\dots,a_d)$, where
\[a_i=\#\{j\in[d] \st i<j \text{ and } \pi_i>\pi_j\}.\]
Note that $a$ is a weak composition of $n$ with $d$ parts, hence
a vertex of $\SR(d,n)$.
A permutation $\sigma\in\Sym_{d,n}$ is
\emph{$\pi$-admissible} if $\sigma$ is a proper skyline rook placement on
$\Sky(a_1+1,\dots,a_d+d)$; that is, if
\[
x(\sigma)=a(\pi)+\soff-\sigma(\soff)=a(\pi)+\id-\sigma
\]
is a lattice point in $n\Delta^{d-1}$.  Note that the coordinates of $x(\sigma)$ sum to $n$,
so admissibility means that its coordinates are all nonnegative.
The set of all $\pi$-admissible permutations is denoted $\Adm(\pi)$; that is,
\[
\Adm(\pi) = \{\sigma\in\Sym_d \st a_i-\sigma_i+i \geq 0 \quad \forall i=1,\dots,d\}.
\]
The corresponding \emph{partial permutohedron}  is
\[
\ParP(\pi) = \{x(\sigma) \st \sigma\in\Adm(\pi)\}.
\]
That is, $\ParP(\pi)$ is the set of permutations corresponding to lattice points
in the intersection of $n\Delta^{d-1}$ with the standard permutohedron centered at
$a(\pi)+\soff$.
The \emph{partial permutohedron vector} is the signed characteristic vector of $\ParP(\pi)$,
that is,
\[
\F_\pi = \sum_{\sigma\in\ParP(\pi)} \sgn(\sigma) \e_{x(\sigma)}.
\]
\end{definition}

\begin{example}
Let $d=4$ and $\pi=3142\in\Sym_d$.  Then $\pi$ has $n=3$ inversions,
namely $12$, $14$, $34$.  Its inversion word is accordingly
$a=(2,0,1,0)$.  The $\pi$-admissible permutations are the proper
skyline rook placements on $\Sky(2+1,0+2,1+3,0+4)=\Sky(3,2,4,4)$,
namely 1234, 1243, 2134, 2143, 3124, 3142, 3214, 3241 (see
Figure~\ref{skyrook-figure}).  The corresponding lattice points
$x(\sigma)$ can be read off from the rook placements by counting the number of
empty squares above each rook, obtaining respectively 2010, 2001, 1110, 1101, 0120, 0102, 
0030, 0003; these are the neighbors of $a$ in $\ParP(\pi)$.  Thus
$\F_\pi= \e_{2010} - \e_{2001} - \e_{1110} + \e_{1101} + \e_{0120} - \e_{0102} - \e_{0030} + \e_{0003}$; see Figure~\ref{partperm-figure}.

\begin{figure}[ht]
\includefigure{2.2in}{0.9in}{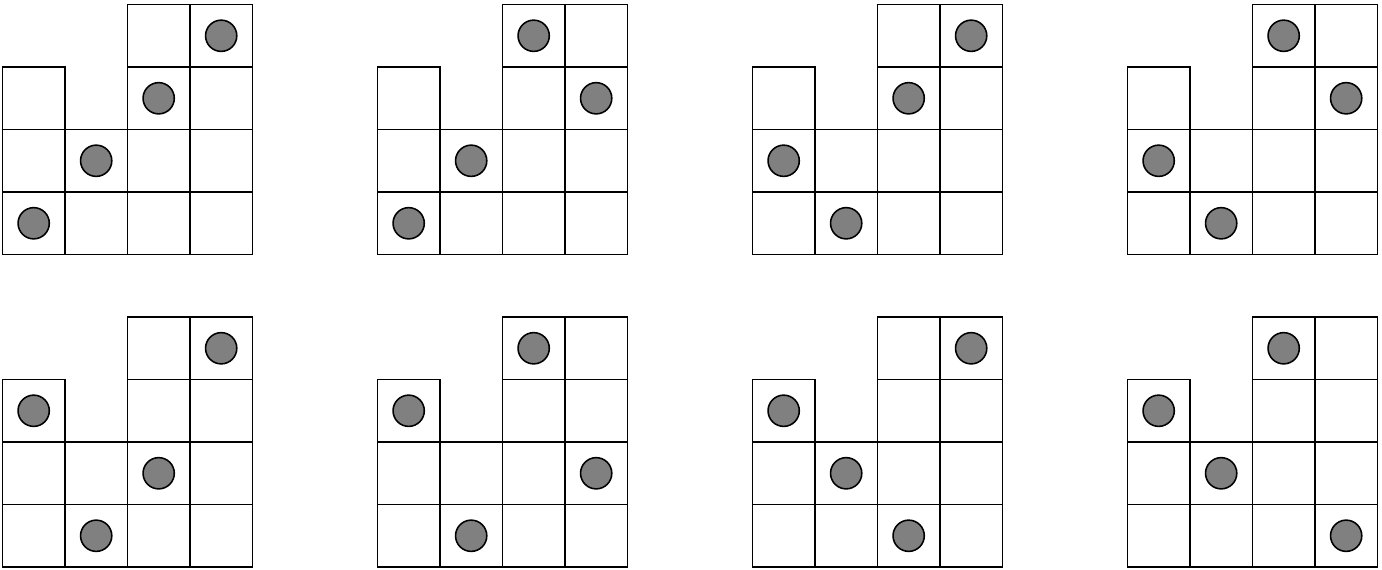} 
\caption{Rook placements on the skyline board $\Sky(3,2,4,4)$. \label{skyrook-figure}}
\end{figure}

\begin{figure}[ht]
\includefigure{2in}{2in}{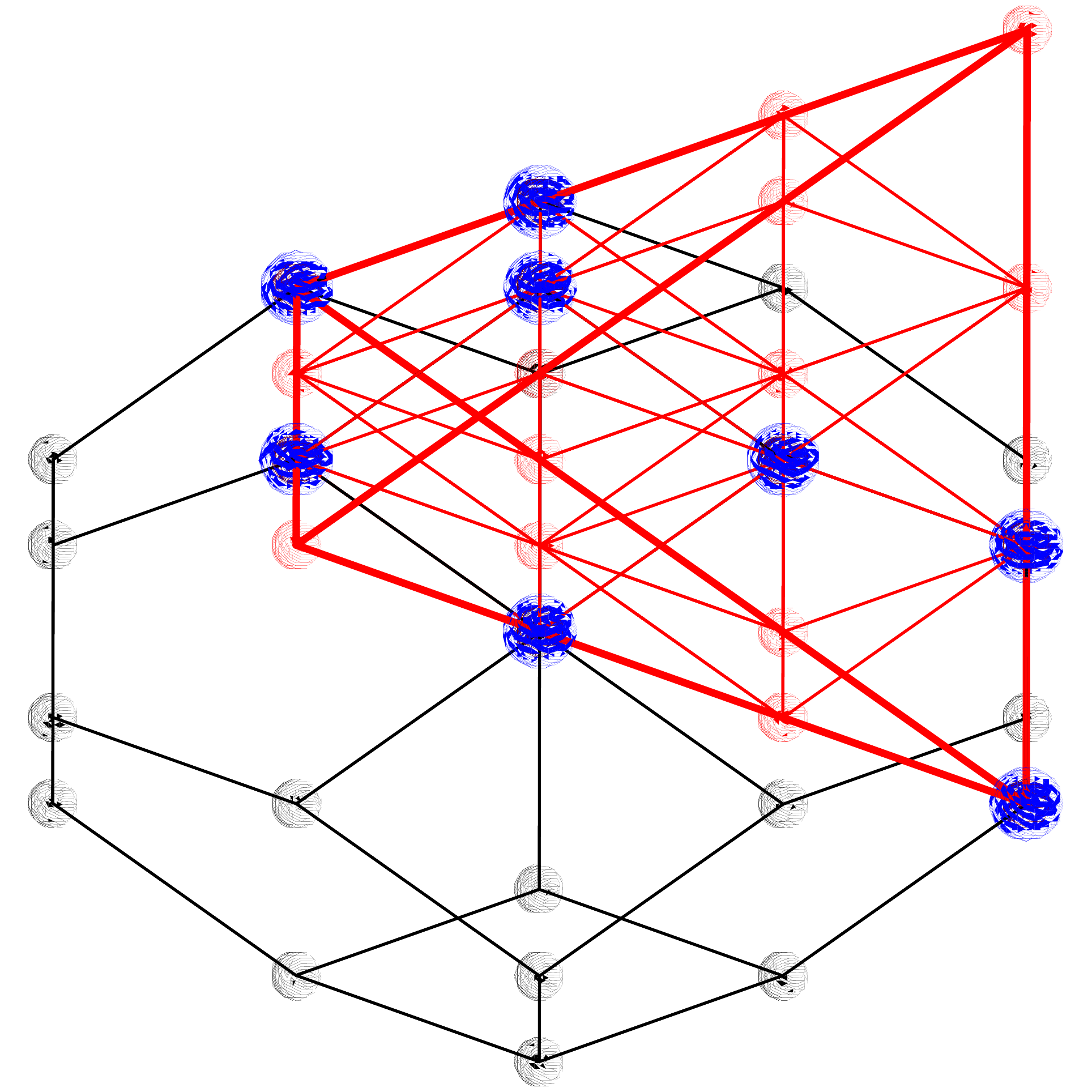}
\caption{The partial permutohedron $\ParP(3142)$ in $\SR(4,3)$.\label{partperm-figure}}
\end{figure}
\end{example}

\begin{theorem} \label{small-n-theorem}
Let $\pi\in\Sym_{d,n}$ and $A=A(d,n)$.
Then~$\F_\pi$ is an eigenvector of~$A$ with eigenvalue~$-n$.
Moreover, for every pair~$d,n$ with $n<\binom{d}{2}$, the set $\{\F_\pi \st \pi\in\Sym_{d,n}\}$ is linearly independent.
In particular, the dimension of the $(-n)$-eigenspace of $A$ is at 
least the Mahonian number $M(d,n)$.
\end{theorem}

\begin{proof}

\textbf{First}, we show that the~$\F_\pi$ are linearly independent.
This follows from the observation that the lexicographically leading term of $\F_\pi$ is $\e_{a(\pi)}$, and
these terms are different for all $\pi\in\Sym_{d,n}$.

\textbf{Second}, let $\sigma\in\Adm(\pi)$.  Then the coefficient of
$\e_{x(\sigma)}$ in~$\F_\pi$ is $\sgn(\sigma)\in\{1,-1\}$.  We will show that the coefficient of
$\e_{x(\sigma)}$ in~$A\F_\pi$ is $-n\sgn(\sigma)$, i.e., that
\begin{equation} \label{F-eigenvalue-n}
\sgn(\sigma) \sum_\rho \sgn(\rho) = -n,
\end{equation}
the sum over all $\rho$ such that $\rho\sim\sigma$ and $\rho\in\ParP(\pi)$.
(Here and subsequently, $\sim$ denotes adjacency in $\SR(d,n)$.)
Each such rook placement $\rho$ is obtained by
multiplying~$\sigma$ by the transposition~$(i\;j)$, that is,
by choosing a rook at $(i,\sigma_i)$, choosing a second rook at $(j,\sigma_j)$ with $\sigma_j>\sigma_i$,
and replacing these two rooks with rooks in positions $(i,\sigma_j)$ and $(j,\sigma_i)$.
For each choice of $i$, there are $(a_i+i)-\sigma_i$ possible $j$'s, and $\sum_i(a_i+i-\sigma_i)=n$.
Moreover, the sign of each such $\rho$ is opposite to that of $\sigma$, proving~\eqref{F-eigenvalue-n}.

\textbf{Third,} let $y=(y_1,\dots,y_d)\in V(d,n)\sm\ParP(\pi)$.
Then the coefficient of
$e_{x(\sigma)}$ in $\F_\pi$ is $0$.
 We will show that the coefficient of
$\e_{x(\sigma)}$ in~$A\F_\pi$ is also 0, i.e., that
\begin{equation} \label{F-eigenvalue-zero}
\sum_{\sigma\in N} \sgn(\sigma) = 0.
\end{equation}
where $N=\{\rho:x(\rho)\sim y\}\cap\ParP(\pi)$.
In order to prove this, we will construct a sign-reversing
involution on~$N$.

Let $a=a(\pi)$ and let
$b=(b_1,\dots,b_d)=(a_1+1-y_1,a_2+2-y_2,\dots,a_d+d-y_d)$.  Note that
$b_i\leq a_i+i$ for every $i$; therefore, we can regard $b$ as a rook
placement on $\Sky(a_1+1,\dots,a_d+d)$.  (It is possible that $b_i\leq
0$ for one or more~$i$; we will consider that case shortly.)  To say
that $y\not\in\F_\pi$ is to say that~$b$ is not a proper $\pi$-skyline
rook placement; on the other hand, we have $\sum b_i = \binom{d+1}{2}$
(as would be the case if $b$ were proper).  Hence the elements of~$N$
are the proper $\pi$-skyline rook skyline placements obtained from~$b$
by moving one rook up and one other rook down, necessarily by the same
number of squares.  Let $\biqjr{i}{q}{j}{r}$ denote the rook placement
obtained by moving the~$i^{th}$ rook up to row~$q$ and the~$j^{th}$
rook down to row~$r$.

We now consider the various possible ways in which $b$ can fail to be
proper.

\Case{1} $b_i\leq 0$ for two or more $i$.  In this case $N=\0$,
because moving only one rook up cannot produce a proper $\pi$-skyline
rook placement.

\Case{2} $b_i\leq 0$ for exactly one $i$.  The other rooks in $b$ cannot
all be at different heights, because that would imply that $\sum
b_i\leq 0+(2+\cdots+d)<\binom{d+1}{2}$.  Therefore, either $N=\0$, or
else $b_j=b_k$ for some $j,k$ and there are
rooks at all heights except $q$ and $r$ for some $q,r<b_j=b_k$.

Then $\biqjr{i}{q}{j}{r}$ is proper if and only if
$\biqjr{i}{q}{k}{r}$ is proper, and likewise
$\biqjr{i}{r}{j}{q}$ is proper if and only if
$\biqjr{i}{r}{k}{q}$ is proper.  Each of these pairs is related by the
transposition $(j\;k)$, so we have the desired sign-reversing involution on~$N$.

\Case{3} $b_i\geq 1$ for all $i$.  Then the reason that $b$ is not proper
must be that some row has no rooks and some row has more than one rook.
There are several subcases:

\Case{3a} For some $q\neq r$, there are two rooks at height $q$, no
rooks at height $r$, and one rook at every other height.  But this is
impossible because then $\sum b_i=\binom{d+1}{2}+q-r\neq\binom{d+1}{2}$.

\Case{3b} There are four or more rooks at height $q$, or three at height $q$
and two or more at height $r$.  In both cases $N=\0$.

\Case{3c} We have $b_i=b_j=b_k$; no rooks at
heights $q$ or $r$ for some $q<r$; and one rook at every other
height.
Then
\[N \subseteq \left\{
\begin{array}{lll}
\biqjr{i}{r}{j}{q},&\biqjr{j}{r}{i}{q},&\biqjr{k}{r}{i}{q},\\
\biqjr{i}{r}{k}{q},&\biqjr{j}{r}{k}{q},&\biqjr{k}{r}{j}{q}.
\end{array}\right\}
\]
For each column of the table above, its two rook placements are related
by a transposition (e.g., $(j\;k)$ for the first column)
and either both or neither
of those rook placements are proper (e.g., for the first
column, depending on whether or not $b_i\leq r$).
Therefore, we have the desired sign-reversing
involution on $N$.

\Case{3d} We have $b_i=b_j=q$; $b_k=b_\ell=r$, and one rook at every
other height except heights $s$ and $t$.  Now the
desired sign-reversing involution on $N$ is
toggling the rook that gets moved down; for instance,
$\biqjr{j}{s}{k}{t}$ is proper if and only if 
$\biqjr{j}{s}{\ell}{t}$ is proper.  

This completes the proof of \eqref{F-eigenvalue-zero},
which together with \eqref{F-eigenvalue-n} completes the proof that
$\F_\pi$ is an eigenvector of $A(d,n)$ with eigenvalue $-n$.
\end{proof}

\begin{conjecture}
If $n\leq\binom{d}{2}$, then in fact $\tau(\SR(d,n))=-n$,
and the dimension of the corresponding eigenspace is the Mahonian
number $M(d,n)$.
\end{conjecture}

We have verified this conjecture, using Sage, for all $d\leq 6$.
It is not clear in general how to rule out the possibility of a smaller
eigenvalue, or of additional ($-n$)-eigenvectors linearly independent
of the $\F_\pi$.

The proof of Theorem~\ref{small-n-theorem} implies that
every partial permutohedron $\ParP(\pi)$ induces an $n$-regular subgraph
of $\SR(d,n)$.  Another experimental observation is the following:
\begin{conjecture}
For every $\pi\in\Sym_{d,n}$, the induced subgraph $\SR(d,n)|_{\ParP(\pi)}$
is Laplacian integral.
\end{conjecture}
We have verified this conjecture, using Sage, for all permutations of
length $d\leq 6$.  We do not know what the eigenvalues are, but these
graphs are not in general strongly regular (as evidenced by the
observation that they have more than 3 distinct eigenvalues).


\section{Corollaries, alternate methods, and further directions}

\subsection{The independence number}
The independence number of $\SR(d,n)$ can be interpreted
as the maximum
number of nonattacking ``rooks'' that can be placed on a simplicial
chessboard of side length $n+1$.  By \cite[Lemma 9.6.2]{GodRoy},
the independence number $\alpha(G)$ of a $\delta$-regular graph $G$ on $N$ vertices
is at most $-\tau N/(\delta-\tau)$, where
$\tau$ is the smallest eigenvalue of $A(G)$.  For $d=3$ and $n\geq 3$, we have $\tau=-3$,
which implies that the independence number $\alpha(\SR(d,n))$ is at most
$3(n+2)(n+1)/(4n+6)$.  This is of course a weaker result (except for a few small values of $n$)
than the exact value $\lfloor(3n+3)/2\rfloor$ obtained in \cite{NAQ} and \cite{Dots}.

\begin{question}
What is the independence number of $\SR(d,n)$?
That is, how many nonattacking rooks can be placed on a simplicial chessboard?
\end{question}

Proposition~\ref{smallest-eigen} implies the upper bound
$$\alpha(\SR(d,n)) \leq \frac{d(d+1)}{(2n+d)(d-1)}\binom{n+d-1}{d-1}$$
for $n\geq\binom{d}{2}$, but this bound is not sharp (for example, the bound for
$\SR(4,6)$ is $\alpha\leq21$, but computation indicates that $\alpha=16$).

\subsection{Equitable partitions}

One approach to determining the spectrum of a graph uses the theory of
\emph{interlacing} and \emph{equitable partitions} \cite{Haemers},
\cite[chapter~9]{GodRoy}.  Let $X=\{O_1,\dots,O_k\}$ be the set of
orbits of vertices of~$G$ under the group of automorphisms of~$G$.
For each two orbits $O_i,O_j$, define $f(i,j) = |N(x)\cap O_j|$ for any
$x\in O_i$.  The choice of $x$ does not matter, so that the function
$f$ is well-defined (albeit not necessarily symmetric); that is to
say, the orbits form an \emph{equitable partition} of $V(G)$.  Let
$P(G)$ be the $k\x k$ square matrix with entries $f(i,j)$.  Then every
eigenvalue of~$P$ is also an eigenvalue of~$A(G)$
\cite[Thm.~9.3.3]{GodRoy}.

When $G=\SR(n,d)$, the spectrum of $P(G)$ is typically a proper subset
of that of $A(G)$.  For example, when $n=3$ and $d=3$, the matrix
$A(G)$ has spectrum $6,1,1,1,0,0,-2,-2,-2,-3$ by
Theorem~\ref{main-theorem}, but the automorphism group has only three
orbits, so $P(G)$ is a $3\x 3$ matrix and must have a strictly smaller
set of eigenvalues.  In fact its spectrum is $6,1,-2$, which is not a
tight interlacing of that of $A(G)$ in the sense of Haemers.

Therefore, these methods may not be sufficient to describe the
spectrum of $\SR(n,d)$ in general.  On the other hand, in all cases we
have checked computationally ($d=4,n\leq 30$; $d=5,n\leq 25$), the
matrices $P(\SR(n,d))$ have integral spectra, which is consistent with
Conjecture~\ref{all-d-n}.

\begin{question} \label{unique-spectrum}
Is $\SR(d,n)$ determined up to isomorphism by its spectrum?
\end{question}

For $\SR(3,3)$, the answer to the question is ``yes,'' for the following reason.  A regular graph is integral if and only if its complement is integral, by \cite[Lemma~8.5.1]{GodRoy}.  Thus the complement $\overline{\SR(3,3)}$ is 3-regular and integral.  There are exactly thirteen such graphs, as classified by Bussemaker, Cvetkovi{\'c}, and Schwenk \cite{MR0465944,MR0409278,MR499520}; see also \cite[pp.~50--51]{SurveyIG}.  Only two of these have ten vertices, namely $\overline{\SR(3,3)}$ and the
Petersen graph, which are not cospectral.
For more on the general problem of which graphs are determined by their spectra, see \cite{MR2022290,MR2499010}.


\section*{Acknowledgements}
We thank Cristi Stoica for bringing our attention to references \cite{NAQ} and \cite{Dots},
and Noam Elkies and other members of MathOverflow for a stimulating discussion.
We also thank an anonymous referee for providing references on Question~\ref{unique-spectrum} and
for suggesting the argument that $\SR(3,3)$ is determined by its spectrum.  The open-source software package Sage \cite{Sage} was a valuable tool in carrying out this research.

\bibliographystyle{amsplain}
\bibliography{biblio2} 
\end{document}